\documentclass[a4paper,11pt]{article}
\usepackage{amsmath,amsthm,amsfonts,amssymb,microtype,mathtools,underscore}
\usepackage[T1]{fontenc}
\usepackage[usenames,dvipsnames,svgnames,table]{xcolor}
\usepackage{breakurl}
\usepackage[unicode=true]{hyperref}
\hypersetup{
colorlinks,
linkcolor={black},
citecolor={black},
urlcolor={blue!60!black},
pdftitle={A general framework for hypergraph colouring},
pdfauthor={Wanless--Wood}}
\usepackage[noabbrev,capitalise]{cleveref}
\crefname{lem}{Lemma}{Lemmas}
\crefname{thm}{Theorem}{Theorems}
\crefname{cor}{Corollary}{Corollaries}
\crefname{prop}{Proposition}{Propositions}
\crefname{conj}{Conjecture}{Conjectures}
\crefname{open}{Open Problem}{Open Problems}

\crefformat{equation}{(#2#1#3)}
\Crefformat{equation}{Equation #2(#1)#3}

\newcommand{\defn}[1]{\textcolor{Maroon}{\emph{#1}}}

\usepackage[longnamesfirst,numbers,sort&compress]{natbib}
\makeatletter
\def\NAT@spacechar{~}
\makeatother

\usepackage[lmargin=30mm,rmargin=30mm,tmargin=30mm,bmargin=30mm]{geometry}
\renewcommand{\baselinestretch}{1.1}
\DeclarePairedDelimiter{\ceil}{\lceil}{\rceil}

\renewcommand{\thefootnote}{\fnsymbol{footnote}} 

\renewcommand{\geq}{\geqslant}
\renewcommand{\leq}{\leqslant}
\renewcommand{\ge}{\geqslant}
\renewcommand{\le}{\leqslant}
\renewcommand{\emptyset}{\varnothing}

\theoremstyle{plain}
\newtheorem{thm}{Theorem}
\newtheorem{cor}[thm]{Corollary}
\newtheorem{lem}[thm]{Lemma}
\newtheorem{conj}[thm]{Conjecture}

\begin{document}

\author{Ian M. Wanless \qquad \qquad David~R.~Wood}

\footnotetext{\today. School of Mathematics, Monash University, Melbourne, Australia. Email: \texttt{\{ian.wanless,david.wood\}@monash.edu}. Research supported by the Australian Research Council.}

\title{\boldmath\bf A general framework for hypergraph colouring}
\date{}
\maketitle

\begin{abstract}
The Lov\'asz Local Lemma is a powerful probabilistic technique for proving the existence of combinatorial objects. It is especially useful for colouring graphs and hypergraphs with bounded maximum degree. This paper presents a general theorem for colouring hypergraphs that in many instances matches or slightly improves upon the bounds obtained using the Lov\'asz Local Lemma. Moreover, the theorem directly shows that there are exponentially many colourings. The elementary and self-contained proof is inspired by a recent result for nonrepetitive colourings by Rosenfeld [2020]. We apply our general theorem in the setting of proper hypergraph colouring, proper graph colouring, independent transversals, star colouring, nonrepetitive colouring, frugal colouring, Ramsey number lower bounds, and for $k$-SAT.
\end{abstract}

\renewcommand{\thefootnote}{\arabic{footnote}}

\section{Hypergraph Colouring}

In their seminal 1975 paper, \citet{EL75} introduced what is now called the Lov\'asz Local Lemma. This tool is one of the most powerful probabilistic techniques for proving the existence of combinatorial objects. Their motivation  was hypergraph colouring. A \defn{hypergraph} $G$ consists of a set $V(G)$ of \defn{vertices} and a set $E(G)$ of \defn{edges}, each of which is a subset of $V(G)$. A \defn{colouring} of a hypergraph $G$ is a function that assigns a `colour' to each vertex of $G$. A colouring of $G$ is \defn{proper} if no edge of $G$ is monochromatic. The \defn{chromatic number $\chi(G)$} is the minimum number of colours in a proper colouring of $G$. The \defn{degree} of a vertex $v$ in a hypergraph $G$ is the number of edges that contain $v$. A hypergraph is \defn{$r$-uniform} if each edge has size $r$. \citet{EL75} proved (using the Lov\'asz Local Lemma) that $\chi(G) \leq \ceil{(4r\Delta)^{1/(r-1)}}$ for every $r$-uniform hypergraph $G$ with maximum degree $\Delta$. The following result is a consequence of the strengthened Lov\'asz Local Lemma first stated by \citet{Spencer77}; see the book by \citet{MR02} for a comprehensive treatment.

\begin{thm}[\citep{EL75,Spencer77}] 
\label{EL} 
For every $r$-uniform hypergraph $G$ with maximum degree $\Delta$,
$$\chi(G) \leq \ceil{(e(r(\Delta-1)+1))^{1/(r-1)}}.$$
\end{thm}


This paper presents a general theorem for colouring hypergraphs, which in the special case of proper hypergraph colouring, (slightly) improves the upper bound in \cref{EL}. Moreover, the proof directly shows that there are exponentially many such colourings. The proof uses a simple counting argument inspired by a recent result for nonrepetitive colourings by \citet{Rosenfeld20}, which in turn is inspired by the power series method for pattern avoidance~\citep{Rampersad11,Ochem16,BW13}. 

It is well known that the proof of \cref{EL} works in the setting of list colourings, which we now introduce. Let $G$ be a hypergraph. A \defn{list-assignment} for $G$ is a function $L$ that assigns each vertex $v$ of $G$ a set $L(v)$, whose elements are called \defn{colours}. If $|L(v)|= c$ for each vertex $v$ of $G$, then $L$ is a \defn{$c$-list-assignment}. An \defn{$L$-colouring} of $G$ is a function $\phi$ such that $\phi(v)\in L(v)$ for each vertex $v$ of $G$. The \defn{choosability} $\chi_{\textup{ch}}(G)$ is the minimum integer $c$ such that $G$ has a proper $L$-colouring for every $c$-list-assignment $L$ of $G$. For a list assignment $L$ of a hypergraph $G$, let $P(G,L)$ be the number of proper $L$-colourings of $G$. 

The following theorem is our first contribution.

\begin{thm}
\label{Main}
For all integers $r\geq 3$ and $\Delta \geq 1$, and for every $r$-uniform hypergraph $G$ with maximum degree $\Delta$, 
$$\chi_{\text{ch}}(G) \leq c := \ceil*{ \Big( \frac{r-1}{r-2} \Big) \big((r-2)\Delta\big)^{1/(r-1)} }.$$
Moreover, for every $c$-list assignment $L$ of $G$, 
$$P(G,L) \geq \big( (r-2)\Delta \big)^{|V(G)|/(r-1)}.$$ 
\end{thm}

We now compare the above-mentioned bounds. Since $(\frac{r-1}{r-2})^{r-2} < e$, it follows that $\big(\frac{r-1}{r-2} \big) \big((r-2)\Delta\big)^{1/(r-1)} < ( e(r-1)\Delta )^{1/(r-1)}$, and assuming $\Delta\geq r-1$, the bound in \cref{Main} is slightly better than the bound in \cref{EL}. The difference is most evident for small $r$. For example, if $r=3$ then the bound in \cref{Main} is $\ceil[\big]{2\sqrt{\Delta}\,}$ compared with $\ceil[\big]{\sqrt{e(3\Delta-2)}\,}$ from \cref{EL}. 

Several researchers have communicated to us that, with a little effort, one can conclude the existence of exponentially many colourings using the Lov\'asz Local Lemma (or other methods), although as far as we are aware no general result of this nature is published. One attraction of our proof is that it gives exponentially many colourings for free. Indeed, this stronger conclusion is a key to enabling the simple proof. See \citep{Harris21,AS16a,Pluhar09,Alon85,AB88,Gebauer13,Cherkashin11,Beck78,MRT77} for more results on colouring hypergraphs with given maximum degree or number of edges, and see \citep{Thomassen07a,Thomassen07b,DMS19,KP18,DS17,Harutyunyan} for other theorems showing the existence of exponentially many colourings in various graph settings.

\cref{Main} is a special case of a more general result that we introduce in the following section. Then, in \cref{Examples},  we apply this general result to a variety of colouring problems, including hypergraph colouring,  graph colouring, independent transversals, star colouring, nonrepetitive colouring, frugal colouring, Ramsey number lower bounds, and $k$-SAT. \cref{Reflections} concludes by comparing our general result with other techniques including the Lov\'asz Local Lemma and entropy compression. 

\section{General Framework}
\label{GeneralFramework}

For a hypergraph $G$ (allowing parallel edges), let $\mathcal{C}_G$ be the set of all colourings $\phi:V(G)\to\mathbb{Z}$. (For concreteness, we assume all colours are integers.)\ For an edge $e$ of $G$, let $\mathcal{C}_e$ be the set of all colourings $\phi:e \to \mathbb{Z}$. An \defn{instance} is a pair $(G,\mathcal{B})$ where $G$ is a hypergraph and $\mathcal{B}=(\mathcal{B}_e\subseteq \mathcal{C}_e: e\in E(G))$. A colouring $\phi\in\mathcal{C}_G$ is $\mathcal{B}$-\defn{bad} if, for some edge $e\in E(G)$, we have that $\phi$ restricted to $e$ is in $\mathcal{B}_e$. Every other colouring in $\mathcal{C}_G$ is \defn{$\mathcal{B}$-good}. For an integer $c\geq 1$, we say $G$ is \defn{$(\mathcal{B},c)$-choosable} if there is a $\mathcal{B}$-good $L$-colouring of $G$ for every $c$-list assignment $L$ of $G$. For a list assignment $L$ of $G$, let $P(G,\mathcal{B},L)$ be the number of $\mathcal{B}$-good $L$-colourings of $G$. 

Fix an instance $(G,\mathcal{B})$ and consider an edge $e$ of $G$. A subset $S\subseteq e$ \defn{determines} $\mathcal{B}_e$ if any two colourings in $\mathcal{B}_e$ that agree on $S$ are identical. For every vertex $v$ in $e$, we assume that $\mathcal{B}_e$ is determined by some subset of $e\setminus\{v\}$. (Consider this assumption to be part of the definition of `instance'.)\ Then define the \defn{weight} of $(v,e)$ to be $|e|-1-|S|$, where $S$ is a minimum-sized subset of $e\setminus\{v\}$ that determines $\mathcal{B}_e$. 
For each vertex $v$ of $G$, let \defn{$E_k(v)$} be the number of pairs $(v,e)$ with weight $k$.

For example, to model proper colouring in an $r$-uniform hypergraph $G$, for each edge $e$ of $G$, let $\mathcal{B}_e$ be the monochromatic colourings in $\mathcal{C}_e$. Then a colouring is $\mathcal{B}$-good if and only if it is proper. For every edge $e$ and every vertex $v$ in $e$, if $w$ is any vertex in $e\setminus\{v\}$, then $\{w\}$ determines $\mathcal{B}_e$, implying that $(v,e)$ has weight $r-2$. 

\begin{thm}
 \label{General}
 Let $(G,\mathcal{B})$ be an instance. Assume there exist a real number $\beta\geq 1$ and an integer $c\geq 1$ such that for every vertex $v$ of $G$,
 \begin{equation}
 \label{Key}
c \geq \beta + \sum_{k\geq0} \beta^{-k} E_k(v)   .
 \end{equation}
 Then $G$ is $(\mathcal{B},c)$-choosable. Moreover, for every $c$-list assignment $L$ of $G$, 
 $$P(G,\mathcal{B},L) \;\geq\; \beta^{|V(G)|}.$$
\end{thm}

Before proving \cref{General} we make a couple of minor observations. 
If $\beta>1$ then \cref{General} guarantees exponentially many $\mathcal{B}$-good colourings. 
If $\beta=1$ then \cref{General} guarantees at least one $\mathcal{B}$-good colouring. 
In most applications $\beta>1$, but on one occasion the case $\beta=1$ is of interest (see \cref{Proper}). 
When applying \cref{General} it is not necessary to determine the weight of a pair exactly; it suffices to determine a lower bound on the weight (because of the $\beta^{-k}$ term in \cref{Key}, where $\beta\geq 1$). 

\cref{General} is an immediate corollary of the following lemma. If $(G,\mathcal{B})$ is an instance 
with $\mathcal{B}=(\mathcal{B}_e:e\in E(G))$, and $H$ is a sub-hypergraph of $G$, then $(H,\mathcal{B})$ refers to the instance $\big(H,(\mathcal{B}_e:e\in E(H))\big)$. Similarly, if $L$ is a list-assignment for $G$, then we consider $L$ (restricted to $V(H)$) to be a list-assignment for $H$. 

\begin{lem}
 \label{GeneralInduction}
 Let $(G,\mathcal{B})$ be an instance. Assume there exist a real number $\beta\geq 1$ and an integer $c\geq 1$ such that $\cref{Key}$ holds for every vertex $v$ of $G$. Then for every $c$-list assignment $L$ of $G$, for every induced sub-hypergraph $H$ of $G$, and for every vertex $v$ of $H$, 
 $$P(H,\mathcal{B},L) \;\geq\; \beta\, P(H-v,\mathcal{B},L).$$
\end{lem}

\begin{proof}
 We proceed by induction on $|V(H)|$.  
 The base case with $|V(H)|=1$ is trivial. 
 Let $H$ be an induced sub-hypergraph of $G$, and assume the claim holds for all induced sub-hypergraphs of $G$ with less than $|V(H)|$ vertices. Let $v$ be any vertex of $H$. Let $X$ be the set of $\mathcal{B}$-bad $L$-colourings of $H$ that are $\mathcal{B}$-good on $H-v$. Then 
 \begin{align}
 \label{F}
 P(H,\mathcal{B},L) \;=\; c \, P(H-v,\mathcal{B},L) \,-\,|X|.
 \end{align}
 We now find an upper bound for $|X|$. For each $L$-colouring $\phi$ in $X$ there is an edge $e\in E(H)$ containing $v$ such that $\phi\in \mathcal{B}_e$ (if there are several options for $e$, fix a choice arbitrarily). Charge $\phi$ to $(v,e)$. Let $X_k$ be the set of colourings in $X$ that are charged to a pair with weight $k$. 
 Consider $\phi$ in $X_k$ charged to $(v,e)$. Let $S$ be a minimum-sized subset of $e\setminus\{v\}$ that determines $\mathcal{B}_e$. Let $T:=e\setminus S$. Then $|T|=k+1$ and $v\in T$. Since $\phi$ is $\mathcal{B}$-good on $H-v$, we know that $\phi$ is also $\mathcal{B}$-good on $H-T$. Since $S$ determines $\mathcal{B}_e$, the number of $L$-colourings in $X_k$ charged to $(v,e)$ is at most  $P(H-T,\mathcal{B},L)$. By induction,
 \begin{align*}
 P(H-v,\mathcal{B},L) \;\geq\; \beta^k \, P(H-T,\mathcal{B},L). 
 \end{align*}
 Thus the number of $L$-colourings in $X_k$ charged to $(v,e)$ is at most $\beta^{-k}\,P(H-v,\mathcal{B},L)$. 
Hence $|X_k| \,\leq\, E_k(v)\,\beta^{-k}\,P(H-v,\mathcal{B},L)$, and
 \begin{align*}
 |X|  
 \;=\;  \sum_{k\geq 0} |X_k|
 \;\leq\; P(H-v,\mathcal{B},L) \; \sum_{k\geq 0} E_k(v) \, \beta^{-k} .
 \end{align*}
 By \cref{F}, 
 \begin{align*}
 P(H,\mathcal{B},L) 
 \;\geq \; & c \, P(H-v,\mathcal{B},L) \,-\, P(H-v,\mathcal{B},L)\; 
 \sum_{k\geq 0} \beta^{-k}  E_k(v).
 \end{align*}
 By \cref{Key},  $P(H,\mathcal{B},L) \,\geq \, \beta\, P(H-v,\mathcal{B},L)$, as desired. 
\end{proof}

\section{Examples}
\label{Examples}

In this section, we apply \cref{General} for various types of (hyper)graph colouring problems and for $k$-SAT. In most cases, \cref{General} matches or improves on the best known bound on the number of colours (as a function of maximum degree), and in addition shows that there are exponentially many colourings. 

\subsection{Proper Colouring}
\label{Proper}

First we prove \cref{Main}. Let $G$ be an $r$-uniform hypergraph with maximum degree $\Delta$ where $r\geq 3$. For each edge $e$ of $G$, let $\mathcal{B}_e$ be the monochromatic colourings in $\mathcal{C}_e$; then a colouring is $\mathcal{B}$-good if and only if it is proper. Each pair $(v,e)$ 
has weight $r-2$, and $E_{r-2}(v)\leq\Delta$. Observe that \cref{Key} holds with $\beta := \big( (r-2)\Delta \big)^{1/(r-1)}$ and
$c:= \ceil[\big]{ \big( \frac{r-1}{r-2} \big) \big((r-2)\Delta\big)^{1/(r-1)} } $. \cref{Main} then follows from \cref{General}. 

Now consider proper colouring in a graph with maximum degree $\Delta$ (the case $r=2$ in the above). Then every pair $(v,e)$ 
has weight 0, and $E_0(v)\leq\Delta$. Thus $c:=\ceil{ \Delta + \beta}$ satisfies \cref{Key}. \cref{General}  with $\beta=1$ says that every graph $G$ with maximum degree $\Delta$ is $(\Delta+1)$-choosable. \cref{General}  with $\beta\geq 2$ says that for every $(\Delta+\beta)$-list assignment $L$ of $G$ there are at least $\beta^{|V(G)|}$ $L$-colourings. These well-known facts are easily proved by a greedy algorithm. It is interesting that the above general framework includes such statements (the Lov\'asz Local Lemma does not). Note that the Local Action Lemma of \citet{Bernshteyn14} is another general-purpose tool that implies $(\Delta+1)$-colourability; also see \citep{Bernshteyn17}. 

See \citep{Rassmann17,Rassmann19} for results about the number of 2-colourings in random 
hypergraphs and about the number of $k$-colourings in random graphs.

\subsection{Star Colouring}

A colouring $\phi$ of a graph $G$ is a \defn{star colouring} if it is proper and every bichromatic subgraph is a star forest; that is, there is no 2-coloured $P_4$ (path on four vertices). The \defn{star chromatic number} $\chi_{\text{st}}(G)$ is the minimum number of colours in a star colouring of $G$. \citet{FRR04} proved (using the Lov\'asz Local Lemma) that $\chi_{\text{st}}(G) \leq O(\Delta^{3/2})$ for every graph $G$ with maximum degree $\Delta$, and that this bound is tight up to a $O(\log\Delta)$ factor. The best known bound  is $\chi_{\text{st}}(G) \leq \sqrt{8}\Delta^{3/2}+\Delta$ proved by \citet{EsperetParreau} using entropy compression. Both these methods work for star choosability. We prove the same bound holds with exponentially many colourings. 

\begin{thm}
\label{StarColouring}
Every graph $G$ with maximum degree $\Delta$ is star $\ceil{\Delta + \sqrt{8\Delta}(\Delta-1)}$-choosable. Moreover, for every $\ceil{\Delta + \sqrt{8\Delta}(\Delta-1)}$-list assignment $L$, there are at least $\big(\sqrt{2\Delta}(\Delta-1)\big)^{|V(G)|}$ star $L$-colourings of $G$.
\end{thm}

\begin{proof}
Define the following hypergraph $G'$ with $V(G')=V(G)$. Introduce one edge $e=\{v,w\}$ to $G'$ for each edge $vw$ of $G$, where $\mathcal{B}_e$ is the set of $L$-colourings $\phi\in\mathcal{C}_e$ such that $\phi(v)=\phi(w)$, and introduce one edge $e=\{u,v,w,x\}$ to $G'$ for each $P_4$ subgraph $(u,v,w,x)$ of $G$, where $\mathcal{B}_e$ is the set of $L$-colourings $\phi\in\mathcal{C}_e$ such that $\phi(u)=\phi(w)$ and $\phi(v)=\phi(x)$. For any list assignment $L$ of $G$, note that $G$ is star $L$-colourable if and only if $P(G',\mathcal{B},L)\geq 1$. Also, the weight of each 2-element edge is 0, and the weight of each 4-element edge is 1. Thus $E_0(v)\leq\Delta$ and $E_1(v) \leq 2\Delta(\Delta-1)^2$. Since \cref{Key} is satisfied with $\beta:=\sqrt{2\Delta}(\Delta-1)$ and  $c:=\ceil{\Delta + \sqrt{8\Delta}(\Delta-1)}$, the result follows from \cref{General}.
\end{proof}

\subsection{Nonrepetitive Graph Colouring}\label{ss:nonrep}

Let $\phi$ be a colouring of a graph $G$. A path $(v_1,\dots,v_{2t})$ in $G$ is \defn{repetitively coloured} by $\phi$ if $\phi(v_i)=\phi(v_{t+i})$ for each $i\in\{1,\dots,t\}$. A colouring $\phi$ of $G$ is \defn{nonrepetitive} if no path in $G$ is repetitively coloured by $\phi$. The \defn{nonrepetitive chromatic number} $\pi(G)$ is the minimum number of colours in a nonrepetitive colouring of $G$. The \defn{nonrepetitive choice number} $\pi_{\text{ch}}(G)$ is the minimum integer $c$ such that $G$ has a nonrepetitive $L$-colouring for every $c$-list assignment $L$ of $G$. \citet{AGHR02} proved that $\pi(G)\leq O(\Delta^2)$ for every graph with maximum degree $\Delta$, and that this bound is tight up to a $O(\log\Delta)$ factor. The proof shows the same bound for $\pi_{\text{ch}}$. Several authors subsequently improved the constant in the $O(\Delta^2)$ term: 
to $36\Delta^2$ by \citet{Grytczuk07}, 
to $16\Delta^2$ by \citet{Gryczuk-IJMMS07}, 
to $(12.2+o(1))\Delta^2$ by \citet{HJ-DM11}, 
and to $10.4\Delta^2$ by \citet{KSX12}. 
All these proofs used the Lov\'asz Local Lemma. 
\citet{DJKW16} improved the constant to 1, by showing that for every graph $G$ with maximum degree $\Delta$,
\begin{equation}
\label{DeltaSquared}
\pi(G) \leq \Delta^2 + O(\Delta^{5/3}).
\end{equation}
The proof of \citet{DJKW16} uses entropy compression; see \citep{GMP14,EsperetParreau} for refinements and simplifications to the method. Equation~\cref{DeltaSquared} was subsequently proved using the Local Cut Lemma of \citet{Bernshteyn17} and using cluster-expansion \citep{BFPS11,Aprile14}. Most recently, \citet{Rosenfeld20}  proved \cref{DeltaSquared} with exponentially many colourings. His paper inspired the present work. We now show that the result of Rosenfeld follows from our general framework. Note that all of the above results hold in the setting of choosability. 

\begin{thm}
For every graph $G$ with maximum degree $\Delta$, if 
$$\beta:= (1+2^{1/3} \Delta^{-1/3})(\Delta-1)^2 \quad \text{ and } \quad
c:= \ceil{ \beta +  2^{-2/3} \Delta^{5/3} (1+ 2^{1/3} \Delta^{-1/3} )^2  },$$
then $G$ is nonrepetitively $c$-choosable. Moreover, for every $c$-list assignment $L$ of $G$ there are at least 
$\beta^{|V(G)|}$ nonrepetitive $L$-colourings of $G$.
\end{thm}

\begin{proof}
Let $G'$ be the hypergraph with $V(G')=V(G)$, where there is an edge $V(P)$ for each path $P$ in $G$ of even order. Here we consider a path to be a subgraph of $G$, so that a path and its reverse contribute one edge to $G'$. For each edge $e$ of $G'$ corresponding to a path $P$ in $G$ of order $2t$, let $\mathcal{B}_e$ be the set of $L$-colourings $\phi\in\mathcal{C}_e$ such that $P$ is repetitively coloured by $\phi$. Thus $G$ is nonrepetitively $L$-colourable if and only if $P(G',\mathcal{B},L)\geq 1$.
 
Consider an edge $e$ of $G'$ corresponding to a path $P$ in $G$ on $2t$ vertices. For each vertex $v$ in $P$, 
any colouring $\phi\in \mathcal{B}_e$ is uniquely determined by $\phi$ restricted to the $t$ vertices in the half of $P$ not containing $v$. Hence $(v,e)$ has weight $t-1$. Every vertex of $G$ is in at most $t\Delta(\Delta-1)^{2t-2}$ paths on $2t$ vertices. So $E_{t-1}(v) \leq t\Delta(\Delta-1)^{2t-2}$. Equation~\cref{Key} requires
$$c \geq \beta + \sum_{t\geq1}  t\Delta(\Delta-1)^{2t-2} \, \beta^{1-t} .$$
Define $\beta:= (1+\epsilon)(\Delta-1)^2$ where $\epsilon>0$ is defined shortly. 
Equation~\cref{Key} requires
$$c \geq (1+\epsilon)(\Delta-1)^2 + 
\Delta\sum_{t\geq 1} t \, (1+\epsilon)^{-t+1} 
= (1+\epsilon)(\Delta-1)^2 + \epsilon^{-2}(1+\epsilon)^2 \Delta .$$ 
Define $\epsilon := 2^{1/3} \Delta^{-1/3}$ (to approximately minimise $(1+\epsilon)(\Delta-1)^2 + \epsilon^{-2}(1+\epsilon)^2 \Delta$). Then \cref{Key} holds with $c$ defined above, 
and the result follows from \cref{General}.
\end{proof}

\subsection{Frugal Colouring}

For an integer $k\geq 1$, a colouring $\phi$ of a graph $G$ is \defn{$k$-frugal} if $\phi$ is proper and $\big|\{w\in N_G(v): \phi(w)=i\} \big| \leq k$ for every vertex $v$ of $G$ and for every colour $i$, where $N_G(v)$ is the set of neighbours of $v$ in $G$. A 1-frugal colouring of $G$ corresponds to a proper colouring of $G^2$. \citet{HMR97} proved that for each integer $k\geq 1$ and sufficiently large $\Delta$, every graph with maximum degree $\Delta$ has a $k$-frugal colouring with $\max\{(k+1)\Delta,\frac{e^3}{k}\Delta^{1+1/k}\}$ colours. An example due to Alon shows that this upper bound is within a constant factor of optimal~\citep{HMR97}. In particular, for all $\Delta\geq k\geq 1$,  Alon constructed a graph with maximum degree at most $\Delta$ that has no $k$-frugal colouring with $\frac{1}{2k}\Delta^{1+1/k}$ colours. Here we improve the constant in the upper bound without assuming that $\Delta$ is sufficiently large, and with exponentially many colourings. 

\begin{thm}\label{Frugal} 
	For all integers $\Delta>k\ge2$, let 
	\begin{equation*}
	\beta:=\left(  (k-1)\Delta \binom{ \Delta-1}{k} \right)^{1/k}
	 \quad \text{and} \quad 
	c:= \Delta + \ceil*{ \frac{k\beta}{k-1} } .
	\end{equation*}
	Then every graph $G$ with maximum degree $\Delta$ has a $k$-frugal $c$-colouring. Moreover, for every
	$c$-list-assignment $L$ of $G$, the number of $k$-frugal
	$L$-colourings of $G$ is at least $\beta^{|V(G)|}$.
\end{thm}

\begin{proof}
	Let $G'$ be the hypergraph with $V(G')=V(G)$, where every edge of $G$ is an edge of $G'$, and $\{w_1,\dots,w_{k+1}\}$ is an edge of $G'$ for every vertex $v$ of $G$ and set $\{w_1,\dots,w_{k+1}\}\subseteq N_G(v)$. In the latter case, we say the edge is \defn{centred} at $v$. For every edge $e=\{v,w\}$ of $G'$, let $\mathcal{B}_e$ be the set of $L$-colourings $\phi\in\mathcal{C}_e$ such that $\phi(v)=\phi(w)$. For every edge $e=\{w_1,\dots,w_{k+1}\}$ of $G'$, let $\mathcal{B}_e$ be the set of $L$-colourings $\phi\in\mathcal{C}_e$ such that $\phi(w_1)=\phi(w_2)=\dots=\phi(w_{k+1})$. Then a colouring of $G$ is $k$-frugal if and only if it is $\mathcal{B}$-good. 
	
	For each edge $e=\{v,w\}$ of $G'$, both $(v,e)$ and $(w,e)$ have weight 0. Consider an edge $e=\{w_1,\dots,w_{k+1}\}$ of $G'$ centred at $v$. For each $i\in\{1,\dots,k+1\}$, the pair $(w_i,e)$ has weight $k-1$, since every colouring $\phi\in\mathcal{B}_e$ is determined by $\{w_j\}$ for any $j\neq i$. 
	
	Consider a vertex $v$ of $G$. Then $E_0(v)\leq\Delta$. Now consider a pair $(v,e)$ with non-zero weight. Then $(v,e)$ has weight $k-1$, and $e=\{w_1,\dots,w_k,v\}$ is centred at some vertex $u$, for some vertices $w_1,\dots,w_k\in N_G(u)\setminus\{v\}$. There are at most $\Delta$ choices for $u$ and at most $\binom{\Delta-1}{k}$ choices for $w_1,\dots,w_k$. Thus $E_{k-1}(v) \leq \Delta \binom{\Delta-1}{k}$. 
	Hence
	\begin{align*}
 \beta + \sum_{i\geq0}  E_i(v) \, \beta^{-i} 
\le  
\beta + \Delta + \Delta \binom{\Delta-1}{k} \, \beta^{1-k}  
	=  \Delta  + \frac{k\beta}{k-1} 	\leq  c.
	\end{align*}  
	The result follows from \cref{General}.
\end{proof}


Since $k (k-1)^{-1+1/k} \to 1$ and $\binom{\Delta-1}{k}^{1/k} \leq \frac{e}{k}(\Delta-1)$, 
\cref{Frugal} implies this:  


\begin{cor}\label{cy:asyfrugal}
As $\Delta> k\rightarrow\infty$, for every $\ceil{(e+o(1))\Delta^{1+1/k}/k}$-list-assignment $L$ of a graph $G$ with maximum degree $\Delta$, the number of $k$-frugal $L$-colourings of $G$ is at least $\beta^{|V(G)|}$.
\end{cor}

Note that Alon's example in \citep{HMR97} shows that \cref{cy:asyfrugal} is within a factor of $2e+o(1)$ of optimal.

\subsection{Independent Transversals and Constrained Colourings}

Consider a hypergraph $G$. A set $X\subseteq V(G)$ is \defn{independent} if no edge of $G$ is a subset of $X$. Consider a partition $V_1,\dots,V_n$ of $V(G)$. A \defn{transversal} of $V_1,\dots,V_n$ is a set $X$ such that $|X\cap V_i|=1$ for each $i$. Let $\ell:V(G)\to\{1,\dots,n\}$ be the function where $\ell(v):=i$ for each vertex $v\in V_i$. For $S\subseteq V(G)$, let $\ell(S):=\{\ell(v):v\in S\}$. 
An edge $e$ of $G$ is \defn{stretched} by $V_1,\dots,V_n$ if $|\ell(e)|=|e|$. 
The following theorem provides a condition that guarantees an independent transversal. 

\begin{thm}
	\label{ISH}
	Fix integers $r\geq 2$ and $t\geq 1$. For an $r$-uniform hypergraph $G$, let $V_1,\dots,V_n$ be a partition of $V(G)$ such that $|V_i| \geq t$ and at most $r^{-r} (r-1)^{r-1} t^{r-1}\,|V_i|$ stretched edges in $G$ intersect $V_i$, for each $i \in\{1,\dots,n\}$. Then there exist at least $(\frac{r-1}{r}t)^n$ independent transversals of $V_1,\dots,V_n$. 
\end{thm}

\begin{proof}
Non-stretched edges do not influence whether a transversal is independent, so we may assume that every edge is stretched.  We may also assume that $|V_i|=t$, since if $|V_i|>t$ and $v$ is a vertex in $V_i$ with maximum degree, then by removing $v$ and its incident edges we obtain another hypergraph satisfying the assumptions. Let $X$ be the hypergraph with $V(X):=\{1,\dots,n\}$, where for each edge $\{v_1,\dots,v_r\}$ of $G$ there is an edge $e=\{\ell(v_1),\dots,\ell(v_r)\}$ in $X$. By assumption, each vertex $i$ of $X$ has degree at most 
$r^{-r} (r-1)^{r-1} t^{r-1}|V_i| = r^{-r} (r-1)^{r-1} t^r$. Let $L$ be the list-assignment of $X$ with $L(i):=V_i$ for each $i\in\{1,\dots,n\}$. For each edge $e$ of $X$ corresponding to edge $\{v_1,\dots,v_r\}$ of $G$, let $\mathcal{B}_e$ be the set consisting of the $L$-colouring $\phi$ of $e$ with $\phi(\ell(v_j))=v_j$ for each $j\in\{1,\dots,r\}$. Thus $\mathcal{B}$-good $L$-colourings of $X$ correspond to independent transversals of $V_1,\dots,V_n$. Since $\mathcal{B}_e$ is determined by $\emptyset$, each pair $(i,e)$ 
has weight $r-1$. Define $\beta:= \frac{r-1}{r}t$. Then 
$$|L(i)| = t = \beta +  \frac{(r-1)^{r-1}\, t^r}{r^r\,\beta^{r-1}}   \geq \beta + \frac{E_{r-1}(i) }{\beta^{r-1}}.$$ 
Thus \cref{Key} holds and the result follows from \cref{General}. 
\end{proof}

\citet{EGL94} study independent transversals in a particular family of sparse hypergraphs. They define an \defn{$[n,k,r]$-hypergraph} to be an $r$-uniform hypergraph $G$ whose vertex set $V(G)$ is partitioned into $n$ sets $V_1, \dots, V_n$, each with $k$ vertices, such that every edge is stretched by $V_1,\dots,V_n$ and for every $r$-element subset $S$ of $\{1,2,\dots,n\}$ there is exactly one edge $e\in E(G)$ such that $\ell(e)=S$. \citet{EGL94} defined $f_r(k)$ to be the maximum integer $n$ such that every $[n,k,r]$-hypergraph has an independent transversal. 
Using the Lov\'asz Local Lemma, they proved that if 
\begin{equation}
\label{TheirEqn}
e \left(\binom{n}{r} - \binom{n-r}{r} \right) < k^r,
\end{equation}
then $f_r(k) \geq n$.
Observe that for every $[n,k,r]$-hypergraph $G$ with partition $V_1,\dots,V_n$, for each $i\in\{1,\dots,n\}$, exactly $\binom{n-1}{r-1}$ edges of $G$ intersect $V_i$. Thus \cref{ISH} implies that if 
\begin{equation}
\label{OurEqn}
\binom{n-1}{r-1} \leq \frac{(r-1)^{r-1} k^{r}}{r^r},
\end{equation}
then $f_r(k)\geq n$. We now compare these last two results. Consider $r$ to be fixed. As $k$ grows, the largest $n$ satisfying \cref{TheirEqn} or \cref{OurEqn} also grows, so we can think of $n$ being large relative to $r$. Then
\begin{align}
&\hspace{-2em}\frac{(r-1)^{r-1}\left[\binom{n}{r}-\binom{n-r}{r}\right]}{r^r\binom{n-1}{r-1}}\nonumber \\
&=\left(\frac{r-1}{r}\right)^{r-1}\frac{n}{r^2}\left[1-\frac{(n-r)!^2}{n!(n-2r)!}\right]\nonumber\\
&=\left(\frac{r-1}{r}\right)^{r-1}\frac{n}{r^2}\left[1-\prod_{i=0}^{r-1}\frac{n-r-i}{n-i}\right]\nonumber\\
&\ge \left(\frac{r-1}{r}\right)^{r-1}\frac{n}{r^2}\left[1-\left(\frac{n-r}{n}\right)^{r}\,\right]\nonumber\\
&=\left(\frac{r-1}{r}\right)^{r-1}\left[1-\binom{r}{2}\frac{1}{n}+\frac{n}{r^2}\sum_{i=2}^{\lceil r/2\rceil}\left(\binom{r}{2i-1}\left(\frac{r}{n}\right)^{2i-1}-\binom{r}{2i}\left(\frac{r}{n}\right)^{2i}\right)\right]\nonumber\\
&\ge \left(\frac{r-1}{r}\right)^{r-1}\left[1-\frac{r^2}{2n}+\frac{n}{r^2}\sum_{i=2}^{\lceil r/2\rceil}\binom{r}{2i-1}\left(\frac{r}{n}\right)^{2i-1}\left(1-\frac{(r-2i+1)r}{2in}\right)\right]\nonumber\\
&\ge\left(\frac{r-1}{r}\right)^{r-1}\left[1-\frac{r^2}{2n}\right]\label{e:bnd}
\end{align}
if $n\geq{r^2}/{4}$. Also $(1-1/r)^{r-1}>1/e$. Hence, if $n$ is sufficiently large relative to $r$, then \cref{e:bnd} will exceed $1/e$, and \cref{OurEqn} implies \cref{TheirEqn}. In other words, our bound on $f_r(k)$ is better when $k$ is sufficiently large relative to $r$. \citet{Yuster-CPC97,Yuster-DM97} used a different argument to get a better bound in the case of graphs ($r=2$). 

\cref{ISH} in the case of graphs says:

\begin{cor}
\label{IS}
Fix an integer $t\geq 1$. For a graph $G$, let $V_1,\dots,V_n$ be a partition of $V(G)$ such that $|V_i| \geq t$ and there are at most $\frac{t}{4}|V_i|$ edges in $G$ with exactly one endpoint in $V_i$, for each $i \in\{1,\dots,n\}$. Then there exist at least $(\frac{t}{2})^n$ independent transversals of $V_1,\dots,V_n$. 
\end{cor}

\cref{IS} immediately implies the following result (since the average degree out of $V_i$ is at most the maximum degree). 

\begin{cor}
 \label{ISbasic}
 For a graph $G$ with maximum degree at most $\Delta$, if $V_1,\dots,V_n$ is a partition of $V(G)$ such that $|V_i| \geq 4\Delta$ for each $i \in\{1,\dots,n\}$, then there exist at least $(2\Delta)^n$ independent transversals of $V_1,\dots,V_n$. 
\end{cor}

We now compare \cref{IS,ISbasic} with the literature. \citet{RW12} proved the weakening of \cref{IS} with $\frac{t}{4}$ replaced by $\frac{t}{2e}$ and with $(\frac{t}{2})^n$ replaced by 1, and \citet{DEKO20} noted that \cref{IS} holds with $(\frac{t}{2})^n$ replaced by $1$ (using different terminology).  Similarly, \citet{Alon94} proved the weakening of \cref{ISbasic} with $4\Delta$ replaced by $2e\Delta$ and with $(2\Delta)^n$ replaced by 1. The proofs of \citet{RW12} and \citet{Alon94} used the Lov\'asz Local Lemma, while the proof of \citet{DEKO20} used the Local Cut Lemma. Using a different method, \citet{Haxell01} proved the strengthening of \cref{ISbasic} with $4\Delta$ replaced by $2\Delta$, but with $(2\Delta)^n$ replaced by 1. The bound here of $2\Delta$ is best possible \citep{BES-DM75,Yuster-DM97}. It is open whether $\frac{t}{4}$ in \cref{IS} can be improved to $\frac{t}{2}$; see \citep{KK20}. See \citep{LS07,Yuster-CPC97,GS20} for more on independent transversals in graphs.

These results are related to the following `constrained colouring' conjecture of \citet{Reed99}:

\begin{conj}[\citep{Reed99}]
 \label{ReedConj}
 Let $L$ be a $(k+1)$-list assignment of a graph $G$ such that for each vertex $v$ of $G$ and colour $c \in L(v)$, there are at most $k$ neighbours $w\in N_G(v)$ with $c\in L(w)$. Then there exists a proper $L$-colouring of $G$. 
\end{conj}

\citet{Haxell01} observed the following connection between constrained colourings and independent transversals. Consider an $f(k)$-list-assignment $L$ of a graph $G$. Let $H$  be the graph with $V(H):=\{(v,c):v\in V(G), c\in L(v)\}$, where $(v,c)(w,c)\in E(H)$ for each edge $vw\in E(G)$ and colour $c\in L(v)\cap L(w)$. Let $H_v:=\{(v,c):c\in L(v)\}$. Then $(H_v:v\in V(G))$ is a partition of $H$ with each $|H_v|\geq f(k)$ such that proper $L$-colourings of $G$ correspond to independent transversals of $(H_v:v\in V(G))$. Now if we assume that for each vertex $v$ and colour $c\in L(v)$ there are at most $k$ neighbours $w\in N_G(v)$ with $c\in L(w)$, then $H$ has maximum degree at most $k$. Hence the above-mentioned result of \citet{Alon94} proves \cref{ReedConj} with $k+1$ replaced by $2ek$ (also proved by \citet{Reed99}), and the above-mentioned result of \citet{Haxell01} proves \cref{ReedConj} with $k+1$ replaced by $2k$. \citet{BH02} disproved \cref{ReedConj}. The best asymptotic result, due to \citet{RS02}, says that for each $\epsilon>0$ there exists $k_0$ such that \cref{ReedConj} holds with $k+1$ replaced by $(1+\epsilon)k$ for all $k\geq k_0$. None of these results conclude that there are exponentially many colourings. \cref{IS} and the above connection by \citet{Haxell01} implies the following result:

\begin{cor}
\label{Constrained} 
Fix an integer $t\geq 2$. Let $L$ be a $t$-list assignment of a graph $G$ such that for each vertex $v$ of $G$, 
 $$ 4\sum_{w\in N_G(v)}\!\! |L(v)\cap L(w) | \,\leq\, t^2.$$ 
Then there exist at least $(\frac{t}{2})^{|V(G)|}$ proper $L$-colourings of $G$.
\end{cor}

Taking $t=4k$ we obtain the following result in the direction of \cref{ReedConj}:

\begin{cor}
 \label{ConstrainedBasic}
 Let $L$ be a $4k$-list assignment of a graph $G$ such that for each vertex $v$ of $G$ and colour $c \in L(v)$, there are at most $k$ neighbours $w\in N_G(v)$ such that $c\in L(w)$. Then there exist at least $(2k)^{|V(G)|}$ proper $L$-colourings of $G$.
\end{cor}

The following  stronger result can also be proved using a variant of \cref{General}. 
  
\begin{thm}
\label{LastTheorem}
Let $L$ be a list-assignment of a graph $G$ such that for every vertex $v$ of $G$,
\begin{equation}
\label{LastAssumption}
|L(v)| \; \geq \; 4 \sum_{w\in N_G(v)} \frac{ |L(v) \cap L(w)| }{ |L(w)| }.
\end{equation}
Then there exist at least $\prod_{v\in V(G)} \frac{|L(v)|}{2}$ proper $L$-colourings.
\end{thm}


\begin{proof}
We proceed by induction on $|V(H)|$ with the following hypothesis: 
for every induced subgraph $H$ of $G$, and for every vertex $v$ of $H$, 
   $$P(H,L) \;\geq\; \frac{|L(v)|}{2}\, P(H-v,L).$$
   (The proof is very similar to that of \cref{GeneralInduction} except that $\beta$ depends on $v$; in particular, $\beta_v=\frac{|L(v)|}{2}$.)\ 
   The base case with $|V(H)|=1$ is trivial. 
   Let $H$ be an induced subgraph of $G$, and assume the claim holds for all induced subgraphs of $G$ with less than $|V(H)|$ vertices. Let $v$ be any vertex of $H$. 
   Let $X$ be the set of improper $L$-colourings of $H$ that are proper on $H-v$. Then 
   \begin{align}
   \label{FF}
   P(H,L) \;=\; |L(v)| \; P(H-v,L) \,-\,|X|.
   \end{align}
   We now find an upper bound for $|X|$. For $w\in N_G(v)$, let $X_w$ be the set of colourings $\phi$ in $X$ such that $\phi(v)=\phi(w)$. 
   Each $L$-colouring in $X$ is in some $X_w$. Thus
   \begin{align*}
   |X| \;\leq\;  \sum_{w\in N_G(v)}\!\!\! |X_w|  \;\leq\;  \sum_{w\in N_G(v)} P(H-v-w,L) \; |L(v)\cap L(w)|.
   \end{align*} 
   By induction, $P(H-v,L) \;\geq\; \frac{|L(w)|}{2} \, P(H-v-w,L)$. 
   Hence
   \begin{align*}
   \label{XXX}
   |X| \leq \sum_{w\in N_G(v)} \!\!\! \frac{2 |L(v)\cap L(w)|}{|L(w)|} \, P(H-v,L) .
   \end{align*} 
   By \cref{FF}
   \begin{align*}
   P(H,L) \;\geq& \; |L(v)| \; P(H-v,L) \,- \sum_{w\in N_G(v)}\!\!\! \frac{2 |L(v)\cap L(w)|}{|L(w)|} \, P(H-v,L).
   \end{align*}
   By \cref{LastAssumption}, $P(H,L) \geq \frac{|L(v)|}{2} \,P(H-v,L)$, as desired. 
  \end{proof}
 
Note that \cref{LastTheorem} immediately implies \cref{Constrained}, taking $|L(v)|=2t$ for each $v$.

\subsection{Ramsey Numbers}
 
 
For integers $k,c\geq 2$, let \defn{$R_c(k)$} be the minimum integer $n$ such that every edge $c$-colouring of $K_n$ contains a monochromatic $K_k$. \citet{Ramsey30} and \citet{ES35} independently proved that $R_c(k)$ exists.  The best asymptotic lower bound on $R_2(k)$ is due to \citet{Spencer75,Spencer77} who proved that 
 \begin{equation}
 \label{RamseySpencerAsymptotics}
 R_2(k)\geq \bigg(\frac{\sqrt{2}}{e}-o(1)\bigg) k\, 2^{k/2}.
 \end{equation}
 More precisely,  \citet{Spencer75,Spencer77} proved that if 
 \begin{equation}
 \label{RamseySpencer}
 e\binom{k}{2}\left( \binom{n-2}{k-2} + 1 \right) < 2^{\binom{k}{2}-1}, 
 \end{equation}
 then there exists an edge 2-colouring of $K_n$ with no monochromatic $K_k$, implying $R_2(k)>n$.  \cref{General} leads to an analogous result with the same asymptotics, but with slightly better lower order terms. For a graph $G$ and integer $k\geq 2$, let \defn{$D_k(G)$} be the maximum, taken over all edges $vw\in E(G)$, of the number of $k$-cliques in $G$ containing $v$ and $w$. 

 \begin{thm}
	\label{RamseyLemma}
	Fix integers $k\geq 3$ and $c\geq 2$. Let $m:=\binom{k}{2}-1$. Then for every graph $G$ with 
\begin{equation}\label{e:Dbound}
  D_k(G) \leq \frac{(m-1)^{m-1}\,c^m}{m^m},
\end{equation}
	there exists an edge $c$-colouring of $G$ with no monochromatic $K_k$. In fact, there exists at least $\big( D_k(G) (m-1) \big)^{|E(G)|/m}$ such colourings. 
\end{thm}
 
  \begin{proof}
 	Let $G'$ be the hypergraph with $V(G'):=E(G)$, where $S\subseteq E(G)$ is an edge of $G'$ whenever $S$ is the edge-set of a $K_k$ subgraph in $G$. For each edge $vw$ of $G$, let $L(vw):=\{1,\dots,c\}$. For each edge $S$ of $G'$, let $\mathcal{B}_S$ be the set of monochromatic $L$-colourings of $S$. Thus $\mathcal{B}$-good  $L$-colourings of $G'$ correspond to edge $c$-colourings of $G$ with no monochromatic $K_k$. Each pair $(v,e)$ 
 	has weight $m-1$, and $E_{m-1}(v)\leq D_k(G)$. Thus \cref{Key} holds if
 	\begin{equation}
 	\label{KeyRamsey}
 	c \geq \beta + D_k(G) \,\beta^{1-m}.
 	\end{equation}
 	To minimise the right-hand side of this expression, define $\beta := \big( D_k(G) \, (m-1) \big)^{1/m}$. Then \cref{e:Dbound} implies
\cref{KeyRamsey}, so the result follows from \cref{General}. 
 \end{proof}
 
 Applying \cref{RamseyLemma} to a complete graph gives the following corollary. 
 
 \begin{cor}
 \label{Ramsey}
 For every integer $k\geq 3$ and $c\geq 2$, if $m:=\binom{k}{2}-1$ and
  $$ \frac{m^m}{(m-1)^{m-1}}  \binom{n-2}{k-2} \leq c^m $$
 then there exists an edge $c$-colouring of $K_n$ with no monochromatic $K_k$, and $R_c(k)>n$. 
\end{cor}
 
 Since $ \frac{m^m}{(m-1)^{m-1}} < em = e\big( \tbinom{k}{2}-1 \big)$,  \cref{Ramsey} is slightly stronger than \cref{RamseySpencer}. While this improvement only changes the implicit lower order term in \cref{RamseySpencerAsymptotics}, we consider it to be of interest, since it suggests a new approach for proving lower bounds on $R_c(k)$. 
 
\subsection{$k$-SAT}

The $k$-SAT problem takes as input a Boolean formula $\psi$ in conjunctive normal form, where each clause has exactly  $k$ distinct literals, and asks whether there is a satisfying truth assignment for $\psi$. The Lov\'asz Local Lemma proves that if each  variable is in at most $\frac{2^k}{ke}$ clauses, then there exists a satisfying truth assignment; see \citep{GMSW09} for a thorough discussion of this topic. The following result (slightly) improves upon this bound (since $\big( \frac{k-1}{k} \big)^{k-1} > \frac{1}{e}$), and moreover, guarantees exponentially many truth assignments. 


\begin{thm}
 \label{kSAT} 
 Let $\psi$ be a Boolean formula in conjunctive normal form, with variables $v_1,\dots,v_n$ and clauses $c_1,\dots,c_m$, each with exactly $k$ literals. Assume that each variable is in at most $\Delta := \frac{2^k}{k} 
 \big( \frac{k-1}{k} \big)^{k-1} $ clauses. Then there exists a satisfying truth assignment for $\psi$. In fact, there are at least $(2-\frac{2}{k})^n$ such truth assignments. 
\end{thm}

\begin{proof}
 Let $G$ be the hypergraph with $V(G)=\{v_1,\dots,v_n\}$ and $E(G)=\{e_1,\dots,e_m\}$, where edge $e_i$ consists of those variables in clause $c_i$. So $G$ is $k$-uniform. Let $L(v_i)=\{0,1\}$ for each vertex $v_i$. Let $\mathcal{B}_{e_i}$ be the set of $L$-colourings of $e_i$ such that $c_i$ is not satisfied. Satisfying truth assignments for $\psi$ correspond to $\mathcal{B}$-good $L$-colourings of $G$. Each pair $(v,e)$ 
 has weight $k-1$. Thus $E_{k-1}(v)\leq\Delta$ and $E_i(v) = 0$ for all $i\neq k-1$. Then \cref{Key} holds with $\beta := 2-\frac{2}{k}$ and $c:=2$. The result follows from \cref{General}.
\end{proof}

Note that \citet{GST16} proved that if each variable is in at most $(1-o(1))\frac{2^{k+1}}{ke}$ clauses, then there exists a satisfying truth assignment, and that this bound is best possible up to the $o(1)$ term; see \citet{Harris21} for further improvements. These results improve upon the bound in \cref{kSAT} by a factor of 2. However, \cref{kSAT} may still be of interest since it gives exponentially many satisfying assignments and is an immediate corollary of our general framework. 

See \citep{CW18} for bounds on the number of satisfying truth assignments in random {$k$}-{SAT} formulas.

\section{Reflection}
\label{Reflections}

We now reflect on \cref{General}, which provides a general framework
for colouring hypergraphs of bounded degree.

First we discuss minimising the number of colours in \cref{General}.
To do so, one
needs to minimise the right hand side of \cref{Key}, which is a
Laurent series $Q(\beta)$ with nonnegative integer coefficients.  We
assume that at least one edge has positive weight, since otherwise
$Q(\beta)$ is linear.  We also assume that the coefficients in
$Q(\beta)$ grow slowly enough that it and its first two derivatives
converge for all $\beta>R$ for some real number $R$.  For example,
when the weight of edges is bounded (which is true in every example in
this paper outside of \cref{ss:nonrep}), we are optimising a Laurent
polynomial, and may take $R=0$.  Now, $Q''(\beta)>0$ for all
$\beta>R$, so we expect a unique minimum for $Q(\beta)$ on the
interval $[R,\infty)$, say at $\beta=\beta_0$.  Since $Q'(1)\le0$ (or
$R>1$), we must have $\beta_0\ge 1$. Even using a value of
$\beta\ne\beta_0$, one still obtains a non-trivial result from
\cref{General}. In fact, choosing $\beta>\beta_0$ may be desirable
if one wants to find conditions under which there are more
colourings than are guaranteed by taking $\beta=\beta_0$.

Compared with the Lov\'asz Local Lemma, \cref{General} has the advantage of directly proving the existence of exponentially many colourings, and often gives slightly better bounds. The proof of \cref{General} is elementary, and as discussed above, \cref{Key} is often easier to optimise than the General Lov\'asz Local Lemma. 

\cref{General} should also be compared with entropy compression, which is a method that arose from the algorithmic proof of the Lov\'asz Local Lemma due to \citet{MoserTardos}. See \citep{BD17,DJKW16,DFMS20,EsperetParreau,GMP20} for examples of the use of entropy compression in the context of graph colouring. We expect that the results in \cref{Examples} can be proved using entropy compression. For example, see \citep[Theorem~12]{GMP14} for a generic graph colouring lemma in a similar spirit to our \cref{General} that is proved using entropy compression. However, we consider the proof of \cref{General} and the proofs of results that apply \cref{General} to be simpler than their entropy compression counterparts, which require non-trivial analytic techniques from enumerative combinatorics. On the other hand, entropy compression has the advantage that it provides an explicit algorithm to compute the desired colouring, often with polynomial expected time complexity. 

It is also likely that our results in \cref{Examples} can be proved using the Local Cut Lemma~\citep{Bernshteyn17} or via cluster expansion \citep{BFPS11}. The advantage of \cref{General} is the simplicity and elementary nature of its proof. See \citep{FLP20,APS21} for results connecting the Lov\'asz Local Lemma, entropy compression, and cluster expansion. 
 
Finally, we mention a technical advantage of the Lov\'asz Local Lemma and of entropy compression. In the setting of hypergraph colouring, the Lov\'asz Local Lemma and entropy compression need only bound the number of edges that intersect a given edge, whereas \cref{General} requires a bound on the number of edges that contain a given vertex (because the proof is by induction on the number of vertices).

\subsection*{Acknowledgements} Thanks to Danila Cherkashin, Ewan Davies, Louis Esperet, David Harris,  Gwena\'el Joret, Ross Kang, Matthieu Rosenfeld and Lutz Warnke  for helpful feedback on an earlier version of this paper. 

	
{\fontsize{10.5pt}{11.3pt}
\selectfont
\let\oldthebibliography=\thebibliography
\let\endoldthebibliography=\endthebibliography
\renewenvironment{thebibliography}[1]{%
 \begin{oldthebibliography}{#1}%
  \setlength{\parskip}{0ex}%
  \setlength{\itemsep}{0ex}%
 }{\end{oldthebibliography}}
\def\soft#1{\leavevmode\setbox0=\hbox{h}\dimen7=\ht0\advance \dimen7
	by-1ex\relax\if t#1\relax\rlap{\raise.6\dimen7
		\hbox{\kern.3ex\char'47}}#1\relax\else\if T#1\relax
	\rlap{\raise.5\dimen7\hbox{\kern1.3ex\char'47}}#1\relax \else\if
	d#1\relax\rlap{\raise.5\dimen7\hbox{\kern.9ex \char'47}}#1\relax\else\if
	D#1\relax\rlap{\raise.5\dimen7 \hbox{\kern1.4ex\char'47}}#1\relax\else\if
	l#1\relax \rlap{\raise.5\dimen7\hbox{\kern.4ex\char'47}}#1\relax \else\if
	L#1\relax\rlap{\raise.5\dimen7\hbox{\kern.7ex
			\char'47}}#1\relax\else\message{accent \string\soft \space #1 not
		defined!}#1\relax\fi\fi\fi\fi\fi\fi}

}
\end{document}